\documentclass[12pt]{amsart}
\usepackage{amssymb}
\usepackage{amsmath, amscd}
\usepackage{amsthm}
\usepackage{amstext}
\usepackage{cite}
\usepackage[table,xcdraw]{xcolor}
\usepackage{ulem}
\usepackage{tikz}
\usepackage{circuitikz}
\usepackage[colorlinks=true, linkcolor=blue,urlcolor=blue]{hyperref}

\newtheorem{theorem}{Theorem}[section]

\newtheorem{proposition}[theorem]{Proposition}
\newtheorem{lemma}[theorem]{Lemma}

\newtheorem{corollary}[theorem]{Corollary}
\theoremstyle{definition}

\begin{document}
	
\author[P. Danchev]{Peter Danchev}
\address{Institute of Mathematics and Informatics, Bulgarian Academy of Sciences, 1113 Sofia, Bulgaria}
\email{danchev@math.bas.bg; pvdanchev@yahoo.com}	
	
\author[G. Karamali]{Gholamreza Karamali}
\address{Department of Basic Sciences, Shahid Sattari Aeronautical University of Science and Technology, Tehran, Iran}
\email{rezakaramali918@gmail.com}

\author[H. Hosseinnezhad]{Hessam Hosseinnezhad}
\address{Department of Basic Sciences, Shahid Sattari Aeronautical University of Science and Technology, Tehran, Iran}
\email{hosseinnezhad@ssau.ac.ir}

\author[O. Hasanzadeh]{Omid Hasanzadeh}
\address{Department of Mathematics, Tarbiat Modares University, 14115-111 Tehran Jalal AleAhmad Nasr, Iran}
\email{o.hasanzade@modares.ac.ir; hasanzadeomiid@gmail.com}

\title[Some results on triangular coefficient matrix rings]{Some results on triangular \\ coefficient matrix rings}

\maketitle

\date{today}



\begin{abstract}
In this paper, we introduce the concept of a {\it triangular coefficient matrix ring} and investigate the structure of its ideals. We then characterize the radicals of the ring \( R_{h}[x]/\langle x^{n} \rangle \) for every positive integer \( n \), where \( R_{h}[x] \) denotes the Hurwitz polynomial ring and \( \langle x^{n} \rangle \) represents the ideal of this ring generated by \( x^{n} \). Furthermore, we explore several properties that are transferred between the base ring \( R \) and the matrix ring \( H_{n}(R) \) which is a proper subring of the triangular coefficient matrix ring.
\end{abstract}

\section{Introduction and Basic Notions}

Throughout the present article, all rings are assumed to be associative and possess an identity element. Some ring-theoretical properties of triangular matrix rings have been intensively investigated by many authors (see, for instance, \cite{h1, h2, n1, n2, n3, p1}). In the current paper, we define the notion of a {\it triangular coefficient matrix ring} as the set of all triangular matrices with point-wise addition and multiplication stated by
\[
(a_{ij})(b_{ij}) = (c_{ij}) ~ \text{for all} ~ a_{ij}, b_{ij} \in R,
\]
where
\[
c_{ij} = \sum_{k=i}^{j} \binom{j-i}{k-i} a_{ik}b_{kj},
\]
for each \( i \leq j \), and \( c_{ij} = 0 \) for each \( i > j \). We, hereafter, denote this ring by \( \hat{H}_{n}(R) \). It can easily be verified by concrete technical calculations that the set \( \hat{H}_{n}(R) \) actually forms a ring with identity \( I_{n}(R) \). We, henceforth, refer to this ring as the {\it triangular coefficient matrix ring} (for further details, see the bottom of page 3672 in \cite{4}).

Consider now the following special set of triangular matrices:
\[
H_{n}(R) = \left\{
\begin{pmatrix}
    a_1 & a_2 & a_3 & \cdots & a_n\\
    0 & a_1 & a_2 & \cdots & a_{n-1}\\
    0 & 0 & a_1 & \cdots & a_{n-2}\\
    \vdots  & \vdots  & \vdots  &\ddots & \vdots\\
    0 & 0 & 0 & \cdots & a_1
\end{pmatrix} \mid a_i \in R \right\}.
\]
It is quite evident that \( H_{n}(R) \) is a subring of the triangular coefficient matrix ring. All elements of \( H_{n}(R) \) can be designed by \( (a_{1}, a_{2}, \ldots, a_{n}) \). The operation addition in \( H_{n}(R) \) is performed component-wise, while the multiplication is defined for every two elements \( A = (a_{1}, a_{2}, \ldots, a_{n})\) and \(B = (b_{1}, b_{2}, \ldots, b_{n}) \in H_{n}(R)\) as follows:

\[
(AB)(m) = \sum_{i=0}^{m-1} \binom{m-1}{i} a_{i+1} b_{m-i-1}, \quad \text{for every } 1 \leq m \leq n.
\]

On the other hand, the concept of the {\it Hurwitz series ring} was introduced by Keigher in \cite{7,8} as a variant of the ring of formal power series. Furthermore, Keigher and Pritchard demonstrated in \cite{9} that Hurwitz series have significant applications in differential algebra and the study of weak normalization.

In this vein, the ring \( R_{H}[x] \) of Hurwitz series over a ring \( R \) is defined as follows: the elements of \( R_{H}[x] \) are functions \( f: \mathbb{N} \to R \), where \( \mathbb{N} \) is the set of all non-negative integers. The operation addition in \( R_{H}[x] \) is performed component-wise, whereas the multiplication is defined for each \( f, g \in R_{H}[x] \) thus:
\[
(fg)(m) = \sum_{i=0}^{m} \binom{m}{i} f(i) g(m-i), \quad \text{for every } m \in \mathbb{N}.
\]
It is straightforward to verify via concrete technical computations that the set \( R_{H}[x] \) actually is a ring with identity \( 1_{R} \).

\medskip

Let \( \text{supp}(f) \) denote the support of \( f \in R_{H}[x] \), that is, \[\text{supp}(f) = \{ i \in \mathbb{N} \mid f(i) \neq 0 \} .\] Also, suppose \( \Pi(f) \) designates the minimal element in \( \text{supp}(f) \), and \( \Delta(f) \) designates the maximal element in \( \text{supp}(f) \) if it, certainly, exists. So, the ring \( R_{h}[x] \) of Hurwitz polynomials over \( R \) is a subring of \( R_{H}[x] \) consisting of all elements \( f \in R_{H}[x] \) for which \( \Delta(f) < \infty \).

\medskip

One inspects that there exists a ring isomorphism \( \phi: R_{h}[x] / \langle x^{n} \rangle \to H_{n}(R) \), defined by
\[
\phi\left( \sum_{i=0}^{n-1} a_{i} x^{i} \right) = (a_{0}, a_{1}, a_{2}, \ldots, a_{n-1}),
\]
where \( a_{i} \in R \). Therefore, \( R_{h}[x] / \langle x^{n} \rangle \cong H_{n}(R) \) holds, where \( \langle x^{n} \rangle \) is the ideal of \( R_{h}[x] \) generated by \( x^{n} \).

\medskip

In this note, we classify the ideals of the ring \( H_{n}(R) \), establishing a one-to-one correspondence between the prime and maximal ideals of \( R \) and those of \( H_{n}(R) \). Furthermore, we illustrate in a series of assertions that the ring \( H_{n}(R) \) is local (respectively, semi-local, matrix local, Jacobson, 2-primal, right stable range one, Dedekind finite, abelian, clean, nil clean, 2-good) if, and only if, the former ring \( R \) is local (respectively, semi-local, matrix local, Jacobson, 2-primal, right stable range one, Dedekind finite, abelian, clean, nil clean, 2-good).

In what follows, unless otherwise stated, all notations and terminology are standard and align with the well-known classical book \cite{10}. We use \( N_{0}(R) \), \( N_{\ast}(R) \), \( L\text{-rad}(R) \), \( N^{\ast}(R) \), and \( \text{rad}(R) \) to denote the Wedderburn radical (i.e., the sum of all nilpotent ideals), the lower nil radical (i.e., the prime radical), the Levitzky radical (i.e., the sum of all locally nilpotent ideals), the upper nil radical (i.e., the sum of all nil ideals), and the Jacobson radical of \( R \), respectively. Moreover, the following well-established inclusions hold:
\[
N_{\ast}(R) \subseteq L\text{-rad}(R) \subseteq N^{\ast}(R) \subseteq \text{rad}(R) \cap \text{Nil}(R).
\]

\section{Maximal and Minimal Ideals of Hurwitz Quotient \\ of Polynomial Rings} \label{sec2-maximal}

In this section, we establish a one-to-one correspondence between the maximal and minimal (left or right) ideals of the ring \( R \) and those of \( H_{n}(R) \). As an application, we determine the Jacobson radical and the Brown-McCoy radical of the ring \( H_{n}(R) \) and, consequently, of \( R_{h}[x] / \langle x^{n} \rangle \).

For the readers' convenience and completeness of the exposition, recall that a ring \( R \) is said to be {\it local} if \( R / \text{rad}(R) \) is a division ring, and \( R \) is said to be {\it semi-local} if the quotient-ring \( R / \text{rad}(R) \) is a semi-simple ring or, equivalently, if the factor-ring \( R / \text{rad}(R) \) is a left Artinian ring.

\medskip

We begin with the following statement, which plays a fundamental role for our further work. As usual, for an arbitrary ring $R$, the symbol $U(R)$ is reserved for its unit group (i.e., the group of all invertible elements in $R$).

\begin{lemma}\label{2.1}
Let \( R \) be a ring. Then, \( U(H_{n}(R)) = (U(R), R, \ldots, R) \).
\end{lemma}

\begin{proof}
If \( (a_{1}, a_{2}, \ldots, a_{n}) \in U(H_{n}(R)) \), then it is readily to verify that \( a_{1} \in U(R) \). To show validity of the inclusion \( (U(R), R, \ldots, R) \subseteq U(H_{n}(R)) \), assume that \( (a_1, a_2, \ldots, a_n) \in H_{n}(R) \) with \( a_{1} \in U(R) \). Hence, there exists \( b_1 \in R \) such that \( a_{1}b_{1} = 1 \) and \( b_{1}a_{1} = 1 \). Defining \( (b_{j}) = (b_{1}, b_{2}, \ldots, b_{n}) \), for each \( 2 \leq j \leq n \), we deduce
\[
b_{j} = -b_{1} \sum_{k=2}^{j} \binom{j-1}{k-1} a_{k} b_{j-k+1}.
\]
However, by definition, we have \( (a_{i})(b_{j}) = (b_{j})(a_{i}) = I_{n} \), thus proving the claim.
\end{proof}

As an immediate consequence, we yield the following.

\begin{corollary}\label{2.2}
Let \( R \) be a ring. Then, the equality is true: \[ U(R_{h}[x] / \langle x^{n} \rangle) = U(R) + Rx + \cdots + Rx^{n-1} + \langle x^{n} \rangle .\]
\end{corollary}

We are now prepared to prove our first main statement.

\begin{theorem}\label{2.3}
Let \( R \) be a ring. A left (resp., right) ideal \( M \) of the ring \( H = H_{n}(R) \) is a maximal left (resp., right) ideal if, and only if, \( M = (M_{0}, R, \ldots, R) \) for some maximal left (resp., right) ideal \( M_{0} \) of \( R \).
\end{theorem}

\begin{proof}
We shall treat only the case of "left ideals" as the other one follows in a way of similarity. To that goal, let \( M_{0} \) be a maximal left ideal of \( R \), and let \( M = (M_{0}, R, \ldots, R) \). We show that \( M \) is a maximal left ideal of \( H \). In fact, it is easy to verify that \( M \) is a left ideal of \( H \). Let \( (a_{i}) \in H \setminus M \). We need to establish that \( M + (a_{i})H = H \). Since \( a_{1} \notin M_{0} \), we derive \( M_{0} + a_{1}R = R \). Thus, \( m + a_{1}r = 1 \) for some \( r \in R \) and \( m \in M_{0} \). Define \( (b_{j}) \in M \) as \( (b_{j}) = (m, -a_{2}r, -a_{3}r, \ldots, -a_{n}r) \). So,

\begin{eqnarray*}
(b_{j}) + (a_{j})(rI_{n}) &=& (m, -a_{2}r, -a_{3}r, \ldots, -a_{n}r) + (a_{1}r, a_{2}r, a_{3}r, \ldots, a_{n}r) \\
    &=& (m + a_{1}r, 0, \ldots, 0) = (1, 0, \ldots, 0).
\end{eqnarray*}

This ensures the desired equality that \( M + (a_{i})H = H \), proving that \( M \) is a maximal left ideal of \( H \), as required.

Conversely, assume \( M \) is a maximal left ideal of \( H \). Putt \( A := \{(a_{i}) \in H \mid a_{1} = 0 \} \). Since \( A^{n} = 0 \), we infer \( A \subseteq \text{rad}(R) \subseteq M \). Set \[ M_0 := \{m \in R \mid (m, m_2, \ldots, m_n) \in M \} .\] Therefore, \( M = (M_{0}, R, \ldots, R) \). To receive that \( M_0 \) is a maximal left ideal of \( R \), given \( M_0 \subseteq N \) for some left ideal \( N \) of \( R \). Consequently, \( M = (M_0, R, \ldots, R) \subseteq (N, R, \ldots, R) \). But, since \( M \) is maximal, \( M_0 \) must also be maximal, as requested.
\end{proof}

As a direct consequence, we extract the following.

\begin{corollary}\label{2.4}
Let \( R \) be a ring. A left (resp., right) ideal \( M \) of the ring \( R_{h}[x] / \langle x^{n} \rangle \) is a maximal left (resp., right) ideal if, and only if, \[ M = M_{0} + Rx + \cdots + Rx^{n-1} + \langle x^{n} \rangle ,\] for some maximal left (resp., right) ideal \( M_{0} \) of \( R \).
\end{corollary}

According to Lemma \ref{2.1} (see also Theorem \ref{2.3}), we can determine the Jacobson radical of \( H_n(R) \) in the following two ways.

\begin{corollary}\label{2.5}
Let \( R \) be a ring. Then, \( \text{rad}(H_{n}(R)) = (\text{rad}(R), R, \ldots, R) \).
\end{corollary}

\begin{proof}
\textbf{Method 1:} Let \( R \) be an arbitrary ring. Hence, \( \text{rad}(R) = \cap M \), where we runs over all maximal left ideals \( M \) of \( R \). Thus, one inspects that
\[
\text{rad}(H_{n}(R)) = \cap_i (M_{i}, R, \ldots, R) = \left(\cap_i M_{i}, R, \ldots, R \right) = (\text{rad}(R), R, \ldots, R).
\]

\textbf{Method 2:} Assume \( (a_{i}) \in H := H_{n}(R) \), where \( a_{1} \in \text{rad}(R) \). We intend to show that \( (a_{i}) \in \text{rad}(H) \). Let \( (b_{j}) \) be an arbitrary element of \( H \). So, \( 1 - b_{1}a_{1} \in U(R) \), because \( a_{1} \in \text{rad}(R) \). Therefore, owing to Lemma \ref{2.1},
\[
I_{n} - (b_{j})(a_{i}) = (1 - b_{1}a_{1}, \ast, \ldots, \ast) \in U(H).
\]
Thus, \( (a_{i}) \in \text{rad}(H) \), as expected.

Conversely, assume \( (a_{i}) \in \text{rad}(H) \). We show that \( a_{1} \in \text{rad}(R) \). Let \( r \) be an arbitrary element of \( R \). Therefore, \( I_{n} - (rI_{n})(a_{i}) \in U(H) \), implying with the aid of Lemma \ref{2.1} that \( 1 - r a_{1} \in U(R) \). Thus, \( a_{1} \in \text{rad}(R) \), as promised.
\end{proof}

What automatically can be deduced is the following consequence.

\begin{corollary}\label{2.6}
Let \( R \) be a ring. Then, \[\text{rad}(R_{h}[x] / \langle x^{n} \rangle) = \text{rad}(R) + Rx + \cdots + Rx^{n-1} + \langle x^{n} \rangle .\]
\end{corollary}

Now, for each subset \( I \) of the ring \( R \), we define
\[
\nabla(I) = \{ (a_{1}, a_{2}, \ldots, a_{n}) \in H_{n}(R) \mid a_{1} \in I \}.
\]
Clearly, \( I \) is a left (resp., right) ideal of \( R \) if, and only if, \( \nabla(I) \) is a left (resp., right) ideal of \( H_{n}(R) \).

\medskip

We, thereby, arrive at the following claims.

\begin{lemma}\label{2.7}
Let \( I \) be an ideal of the ring \( R \). Then, the isomorphism \( H_{n}(R) / \nabla(I) \cong R / I \) is fulfilled.
\end{lemma}

\begin{proof}
Define the map \( \phi: H_{n}(R) \to R / I \) by \( \phi((a_{i})) = a_{1} + I \). It is plainly verified that \( \phi \) is a surjective homomorphism. Moreover, one sees that
\[
\text{ker}(\phi) = \{ (a_{i}) \in H_{n}(R) \mid a_{1} + I = I \} = \{ (a_{i}) \in H_{n}(R) \mid a_{1} \in I \} = \nabla(I).
\]
Therefore, \( H_{n}(R) / \nabla(I) \cong R / I \), as needed.
\end{proof}

Analogously to the proof of Theorem \ref{2.3}, we can prove the following two assertions.

\begin{proposition}\label{2.8}
The ideal \( M \) of the ring \( H_{n}(R) \) is a maximal ideal if, and only if, \( M = (M_{0}, R, \ldots, R) \) for some maximal ideal \( M_{0} \) of \( R \).
\end{proposition}

\begin{corollary}\label{2.9}
The ideal \( M \) of the ring \( R_{h}[x] / \langle x^{n} \rangle \) is a maximal ideal if, and only if, \( M = M_{0} + Rx + \cdots + Rx^{n-1} + \langle x^{n} \rangle \) for some maximal ideal \( M_{0} \) of \( R \).
\end{corollary}

Traditionally, the Brown-McCoy radical \( \text{rad}'(R) \) of \( R \) is the intersection of all maximal ideals of \( R \). Now, in regard to Proposition \ref{2.8}, we can connect the Brown-McCoy radical of the ring \( R \) and the radical of the ring \( H_{n}(R) \) as recorded below.

\begin{corollary}\label{2.10}
Let \( R \) be a ring. Then, we have the following two statements:\\
\text{(1)}\;\; \( \text{rad}'(H_{n}(R)) = (\text{rad}'(R), R, \ldots, R) \).\\
\text{(2)}\;\; \( \text{rad}'(R_{h}[x] / \langle x^{n} \rangle) = \text{rad}'(R) + Rx + \cdots + Rx^{n-1} + \langle x^{n} \rangle \).
\end{corollary}

We also want to document the following preliminary technicality.

\begin{lemma}\label{2.11}
Let \( R \) be a ring. Then, the isomorphism \( H_{n}(R) / \text{rad}(H_{n}(R)) \cong R / \text{rad}(R) \) holds.
\end{lemma}

\begin{proof}
The claim follows combining Corollary \ref{2.5} and Lemma \ref{2.7}.
\end{proof}

Remember that a ring \( R \) is standardly called {\it local} if \( R / \text{rad}(R) \) is a division ring, and \( R \) is {\it semi-local} if \( R / \text{rad}(R) \) is a left Artinian ring (i.e., a semi-simple ring). Likewise, a ring \( R \) is said to be {\it matrix local} if \( R / \text{rad}(R) \) is a simple Artinian ring.

\medskip

With the preceding lemma at hand, we obtain the following.

\begin{corollary}\label{2.12}
Let \( R \) be a ring. Then, we have the following three statements:\\
\text{(1)}\;\; \( R \) is a local ring if, and only if, \( H_{n}(R) \) is a local ring.\\
\text{(2)}\;\; \( R \) is a semi-local ring if, and only if, \( H_{n}(R) \) is a semi-local ring.\\
\text{(3)}\;\; \( R \) is a matrix local ring if, and only if, \( H_{n}(R) \) is a matrix local ring.
\end{corollary}

\begin{proof}
The proof follows at once from Lemma \ref{2.11}.
\end{proof}

We now concentrate on establishing the following major result.

\begin{theorem}
Let \( R \) be a ring. A left (resp., right) ideal \( I \) of the ring \( H = H_{n}(R) \) is a minimal left (resp., right) ideal if, and only if, \( I = (0, 0, \ldots, 0, I_{0}) \) for some minimal left (resp., right) ideal \( I_0 \) of \( R \).
\end{theorem}

\begin{proof}
We shall deal only with the left case, because the right one follows analogically. To that target, let \( I_0 \) be a minimal left ideal of \( R \). It is not too hard to check that \( I = (0, 0, \ldots, 0, I_{0}) \) is a minimal left ideal of \( H \), as stated.

Conversely, assume that \( I \) is a minimal left ideal of \( H \). We show that, for every \( (a_1, a_2, \ldots, a_n) \in I \), it must be that \( a_{1} = 0 \). Assume, for a possible contradiction, that there exists \( A = (a_1, a_2, \ldots, a_n) \in I \) such that \( a_1 \neq 0 \). Put \( I_{1} := \{b \in R \mid (b, b_2, \ldots, b_n) \in I\} \). Apparently, \( I_1 \) is a non-zero left ideal of \( R \). Also, as \( I \) is a left ideal of \( H \), we detect that \[ 0 \neq H(0, \ldots, 0, 1)I_1 \subseteq I .\] However, \( I \) is a minimal left ideal of \( H \), so \( H(0, \ldots, 0, 1)I_1 = I \). But, \( A \notin H(0, \ldots, 0, 1)I_1 \), leading to the wanted contradiction. Thus, \( a_1 = 0 \). By the same argument, one finds that \( a_1 = a_2 = \ldots = a_{n-1} = 0 \). Therefore, \( I_1 = \{b \in R \mid (0, \ldots, 0, b) \in I\} \), and hence \( I_1 \) is a minimal left ideal of \( R \). Finally, \( I = (0, 0, \ldots, 0, I_{0}) \), as formulated.
\end{proof}

As a valuable consequence, we have the following.

\begin{corollary}
Let \( R \) be a ring. Then, the next two statements hold:\\
\text{(1)}\;\; \( \text{soc}(H_{n}(R)_{H_{n}(R)}) = \text{soc}(0, \ldots, 0, \text{soc}(R_{R})) \).\\
\text{(2)}\;\; \( \text{soc}(_{H_{n}(R)}H_{n}(R)) = \text{soc}(0, \ldots, 0, \text{soc}(_{R}R)) \).
\end{corollary}

\section{Prime and Semi-prime Ideals of Hurwitz Quotient \\ of Polynomial Rings} \label{sec-prim}

In this section, we establish a one-to-one correspondence between the prime (resp., semi-prime) ideals of \( R \) and those of \( H_{n}(R) \). As a non-trivial application, we classify the lower nil-radical, the upper nil-radical, and the Levitzki radical of the ring \( H_{n}(R) \) and, consequently, of \( R_{h}[x] / \langle x^{n} \rangle \).

\medskip

Our first chief result in this direction is the following one.

\begin{theorem} \label{3.1}
An ideal \( P \) of the ring \( H = H_{n}(R) \) is a prime (resp., semi-prime, minimal prime, completely prime, completely semi-prime) ideal if, and only if, \( P = (P_{0}, R, \ldots, R) \) for some prime (resp., semi-prime, minimal prime, completely prime, completely semi-prime) ideal \( P_{0} \) of \( R \).
\end{theorem}

\begin{proof}
We prove only the case for prime ideals as the proofs for the other cases are rather similar. To that goal, let \( P_{0} \) be a prime ideal of \( R \), and let \( P = (P_{0}, R, \ldots, R) \). We show that \( P \) is a prime ideal of \( H \). Indeed, let \( (a_{1}, a_{2}, \ldots, a_{n}) \) and \( (b_{1}, b_{2}, \ldots, b_{n}) \) be elements of \( H \) such that \[ (a_{1}, a_{2}, \ldots, a_{n})H(b_{1}, b_{2}, \ldots, b_{n}) \subseteq P .\] Since \( P_{0} \) is a prime ideal of \( R \) and \( a_{1}Rb_{1} \subseteq P_{0} \), we conclude that \( a_{1} \in P_{0} \) or \( b_{1} \in P_{0} \). Thus, \( (a_{1}, a_{2}, \ldots, a_{n}) \in P \) or \( (b_{1}, b_{2}, \ldots, b_{n}) \in P \) whence \( P \) is a prime ideal of \( H \), as asked for.

Conversely, assume that \( P \) is a prime ideal of \( H \). Since \[ (0, 0, \ldots, 1)H(0, 0, \ldots, 1) = 0 \in P ,\] we have \( (0, 0, \ldots, 1) \in P \). Now, \( (r, 0, \ldots, 0)(0, 0, \ldots, 1) \in P \) for each \( r \in R \), so that \( (0, 0, \ldots, R) \subseteq P \). Similarly, \[ (0, 0, \ldots, 1, 0)H(0, 0, \ldots, 1, 0) \subseteq (0, 0, \ldots, R) \subseteq P ,\] forcing \( (0, 0, \ldots, 1, 0) \in P \). Continuing this process, we obtain \( (0, R, \ldots, R, R) \in P \). Defining \( P_{0} = \{a \in R \mid (a, a_2, \ldots, a_n) \in P\} \), it must be that \( P = (P_{0}, R, \ldots, R) \). Besides, it is routinely checked that \( P_{0} \) is an ideal of \( R \). Letting \( aRb \subseteq P_{0} \) for \( a, b \in R \), we discover that \[ (a, 0, \ldots, 0)H(b, 0, \ldots, 0) \subseteq P ,\] and hence \( (a, 0, \ldots, 0) \in P \) or \( (b, 0, \ldots, 0) \in P \). Therefore, \( a \in P_{0} \) or \( b \in P_{0} \), showing that \( P_{0} \) is a prime ideal of \( R \), as pursued.
\end{proof}

A series of three consequences is given in the sequel.

\begin{corollary}
An ideal \( P \) of the ring \( R_{h}[x] / \langle x^{n} \rangle \) is a prime (resp., semi-prime, minimal prime, completely prime, completely semi-prime) ideal if, and only if, \[ P = P_{0} + Rx + \cdots + Rx^{n-1} + \langle x^{n} \rangle \] for some prime (resp., semi-prime, minimal prime, completely prime, completely semi-prime) ideal \( P_{0} \) of \( R \).
\end{corollary}

\begin{corollary} \label{3.3}
Let \( R \) be a ring. Then, \( N_{\ast}(H_{n}(R)) = (N_{\ast}(R), R, \ldots, R) \).
\end{corollary}

\begin{proof}
\textbf{Method 1:} Let \( R \) be an arbitrary ring. Then \( N_{\ast}(R) = \cap_i P_i \), running over all prime ideals \( P_i \) of \( R \). So, one has that
\[
N_{\ast}(H_{n}(R)) = \cap_i (P_{i}, R, \ldots, R) = \left(\cap_i P_{i}, R, \ldots, R \right) = (N_{\ast}(R), R, \ldots, R).
\]

\textbf{Method 2:} Let \( N = (N_{\ast}(R), R, \ldots, R) \). It is now elementarily to verify that \( N \) is an ideal of \( H := H_{n}(R) \) such that \( H / N \cong R / N_{\ast}(R) \) is semi-prime. Thus, \( N_{\ast}(H) \subseteq N \) is true. To show equality, notice that, for any prime ideal \( P \) of \( H \), we have \( P = (P_{0}, R, \ldots, R) \), where \( P_{0} \) is a prime ideal of \( R \) thanks to Theorem \ref{3.1}. But, since \( N_{\ast}(R) \subseteq P_{0} \), it obviously follows that \( N \subseteq P \). Finally, \( N_{\ast}(H) = N \), as desired.
\end{proof}

\begin{corollary}
Let \( R \) be a ring. Then, \[ N_{\ast}(R_{h}[x] / \langle x^{n} \rangle) = N_{\ast}(R) + Rx + \cdots + Rx^{n-1} + \langle x^{n} \rangle .\]
\end{corollary}

As it is known in the existing literature, a ring \( R \) is called \textit{2-primal} provided that \( N_{\ast}(R) = \text{Nil}(R) \). It is obvious that commutative rings and reduced rings (i.e., rings without non-zero nilpotent elements) are always 2-primal. In \cite{13}, Shin proved that a ring \( R \) is 2-primal if, and only if, every minimal prime ideal of \( R \) is completely prime. Besides, Birkenmeier-Heatherly-Lee showed in \cite{2} the interesting fact that the polynomial ring \( R[x] \) over a 2-primal ring \( R \) is also 2-primal. 

\medskip

We now present the following slightly more general assertion.

\begin{lemma}
The ring \( H_{n}(R) \) is \text{2-primal} if, and only if, \( R \) is a \text{2-primal} ring.
\end{lemma}

\begin{proof}
Suppose first that \( H_{n}(R) \) is 2-primal, and choose \( a \in \text{Nil}(R) \). Then, it follows that \( (a, 0, \ldots, 0) \in Nil(H_{n}(R)) \). Since \( H_{n}(R) \) is 2-primal, it must be that \( (a, 0, \ldots, 0) \in N_{\ast}(H_{n}(R)) \). Invoking Corollary \ref{3.3}, we deduce \( a \in N_{\ast}(R) \), insuring that \( R \) is 2-primal.

Conversely, assume \( R \) is 2-primal, and choose \( (a_{1}, a_{2}, \ldots, a_{n}) \in \text{Nil}(H_{n}(R)) \). Thus, \( (a_{1}, a_{2}, \ldots, a_{n})^{n} = 0 \) for some positive integer \( n \). This assures \( a_{1}^{n} = 0 \), so \( a_{1} \in N_{\ast}(R) \) because \( R \) is 2-primal. Utilizing Corollary \ref{3.3}, one finds that \( (a_{1}, a_{2}, \ldots, a_{n}) \in N_{\ast}(H_{n}(R)) \). Hence, \( H_{n}(R) \) is 2-primal, as asserted.
\end{proof}

We now directly obtain the following.

\begin{corollary}
The ring \( R_{h}[x] / \langle x^{n} \rangle \) is \text{2}-primal if, and only if, \( R \) is a \text{2}-primal ring.
\end{corollary}

It is well known in the existing research sources that a ring \( R \) is said to be a \textit{Jacobson ring} if every prime ideal of \( R \) is semiprimitive. In \cite{15} Watters established that the polynomial ring \( R[x] \) over a Jacobson ring \( R \) is also a Jacobson ring.

\medskip

We now have the following.

\begin{corollary}
The ring \( H_{n}(R) \) is Jacobson if, and only if, \( R \) is a Jacobson ring.
\end{corollary}

Recall also that a ring \( R \) is \textit{semi-primary} if \( \text{rad}(R) \) is nilpotent and \( R / \text{rad}(R) \) is semi-simple.

\begin{lemma}
The ring \( H_{n}(R) \) is semi-primary if, and only if, \( R \) is a semi-primary ring.
\end{lemma}

\begin{proof}
Consulting with Corollary \ref{2.5}, one knows that \( \text{rad}(R) \) is nilpotent if, and only if, \( \text{rad}(H_{n}(R)) \) is nilpotent. Additionally, Lemma \ref{2.7} enables us that \( R / \text{rad}(R) \) is a semi-simple ring if, and only if, \( H_{n}(R) / \text{rad}(H_{n}(R)) \) is a semi-simple ring. In conclusion, the result follows immediately from these two observations.
\end{proof}

\begin{lemma}
Let \( R \) be a ring. Then, we have the following three statements:\\
\text{(1)}\;\; \( N^{*}(H_{n}(R)) = (N^{*}(R), R, \ldots, R) \).\\
\text{(2)}\;\; \( N_{0}(H_{n}(R)) = (N_{0}(R), R, \ldots, R) \).\\
\text{(3)}\;\; \( L\text{-rad}(H_{n}(R)) = (L\text{-rad}(R), R, \ldots, R) \).
\end{lemma}

\begin{proof}
We need to prove only that \( N^{*}(H_{n}(R)) = (N^{*}(R), R, \ldots, R) \), as the evidences for the other cases are similar. To that purpose, given \( A := \{(a_{1}, a_{2}, \ldots, a_{n}) \mid a_{1} = 0\} \). Since it is simply checked that \( A \) is a nil-ideal of \( H_{n}(R) \), we can get \( A \subseteq N^{*}(H_{n}(R)) \). Setting \[ N := \{a \in R \mid (a, a_{2}, \ldots, a_{n}) \in N^{*}(H_{n}(R))\} ,\] it is apparent that \( N \) is an ideal of \( R \), and so \( N^{*}(H_{n}(R)) = (N, R, \ldots, R) \). We now claim that \( N = N^{*}(R) \). In fact, choose \( b \in N \). Then, for each \( p, q \in R \), we have \( pbq \in N \), whence \( (pbq, *, \ldots, *) \in N^{*}(H_{n}(R)) \). Therefore, \( (pbq, *, \ldots, *)^{k} = 0 \) for some positive integer \( k \) guaranteeing that \( (pbq)^{k} = 0 \). Thus, \( b \in N^{*}(R) \), as claimed. 

Conversely, put \[ J := \{(a, 0, \ldots, 0) \mid a \in N^{*}(R)\} .\] Since it is plainly inspected that \( J \) is a nil-ideal of \( H_{n}(R) \), we may get \( J \subseteq N^{*}(H_{n}(R)) \). So, \( (bI_{n}) \in J \subseteq N^{*}(H_{n}(R)) \) for \( b \in N^{*}(R) \), as required.
\end{proof}

Our next consequence is the following one.

\begin{corollary}
Let \( R \) be a ring. Then, we have the following three statements:\\
\text{(1)}\;\; \( N^{*}(R_{h}[x] / \langle x^{n} \rangle) = N^{*}(R) + Rx + \cdots + Rx^{n-1} + \langle x^{n} \rangle \).\\
\text{(2)}\;\; \( N_{0}(R_{h}[x] / \langle x^{n} \rangle) = N_{0}(R) + Rx + \cdots + Rx^{n-1} + \langle x^{n} \rangle \).\\
\text{(3)}\;\; \( L\text{-rad}(R_{h}[x] / \langle x^{n} \rangle) = L\text{-rad}(R) + Rx + \cdots + Rx^{n-1} + \langle x^{n} \rangle \).
\end{corollary}

\section{Further Results} \label{sec-res}

Remember the classical notion that a ring \( R \) satisfies the \textit{left stable range one condition} if, for any \( a, b \in R \) with \( Ra + Rb = R \), there exists \( c \in R \) such that \( a + cb \in U(R) \). The stable range one condition is particularly interesting due to the well-known Evans' theorem from \cite{6}, which states that a module \( M \) cancels from direct sums whenever \( \text{End}_{R}(M) \) has left stable range one.

\medskip

We are now ready to establish the following subsequent facts, which give some satisfactory criteria.

\begin{lemma} \label{lemma4.1}
Let \( R \) be a ring. Then, \( H = H_{n}(R) \) has the left stable range one if, and only if, \( R \) has the left stable range one.
\end{lemma}

\begin{proof}
Given \( A = (a, a_{2}, \ldots, a_{n}) \) and \( B = (b, b_{2}, \ldots, b_{n}) \in H \) such that \( HA + HB = H \). Thus, \( Ra + Rb = R \), and since \( R \) has left stable range one, there exists \( c \in R \) such that \( a + cb \in U(R) \). So, \( A + (cI_{n}(R))B \in U(H) \) in virtue of Lemma \ref{2.1}.

Conversely, assume \( a, b \in R \) such that \( Ra + Rb = R \). Then, it is obvious that \( H(a, 0, \ldots, 0) + H(b, 0, \ldots, 0) = H \), and since \( H \) has left stable range one, there exists \( C = (c, c_{2}, \ldots, c_{n}) \in H \) such that \( (a, 0, \ldots, 0) + C(b, 0, \ldots, 0) \in U(H) \). Consequently, \( a + cb \in U(R) \) bearing in mind Lemma \ref{2.1}, as requested.
\end{proof}

We now may extract the following consequence.

\begin{corollary}
Let \( R \) be a ring. Then, \( R_{h}[x] / \langle x^{n} \rangle \) has the left stable range one if, and only if, \( R \) has the left stable range one.
\end{corollary}

Usually, a ring \( R \) is said to be \textit{Dedekind finite}, provided that \( ab = 1 \) implies \( ba = 1 \) for any \( a, b \in R \).

\begin{lemma}
Let \( R \) be a ring. Then, \( H := H_{n}(R) \) is Dedekind finite if, and only if, \( R \) is Dedekind finite.
\end{lemma}

\begin{proof}
Choose \( A = (a_{1}, a_{2}, \ldots, a_{n}) \) and \( B = (b_{1}, b_{2}, \ldots, b_{n}) \in H \) such that \( AB = I_{n} \). This gives the following system of equations:
\begin{eqnarray}
a_{1}b_{1} &=& 1, \label{4.1} \\
a_{1}b_{2} + a_{2}b_{1} &=& 0, \label{4.2} \\
a_{1}b_{3} + 2a_{2}b_{2} + a_{3}b_{1} &=& 0, \label{4.3} \\
&\cdots& \nonumber
\end{eqnarray}

Since \( R \) is Dedekind finite, from (\ref{4.1}), we have \( b_{1}a_{1} = 1 \). Multiplying (\ref{4.2}) on the left side by \( b_{1} \) and on the right side by \( a_{1} \), we obtain \( b_{2}a_{1} + b_{1}a_{2} = 0 \), which guarantees that \( b_{2} = -b_{1}a_{2}b_{1} \) (denoted as \textbf{Eq 1}). 

Similarly, multiplying (\ref{4.3}) on the left side by \( b_{1} \), we get:
\[
b_{3} + 2b_{1}a_{2}b_{2} + b_{1}a_{3}b_{1} = 0.
\]
Substituting \textbf{Eq 1} into this equation leads to the equality
\[
b_{3} - 2b_{1}a_{2}b_{1}a_{2}b_{1} + b_{1}a_{3}b_{1} = b_{3} + 2b_{2}a_{2}b_{1} + b_{1}a_{3}b_{1} = 0.
\]
Furthermore, multiplying from the right side by \( a_{1} \), we arrive at \( b_{3}a_{1} + 2b_{2}a_{2} + b_{1}a_{3} = 0 \). Continuing this process, we can conclude that \( BA = I_{n} \), as needed. 

Conversely, if \( H \) is Dedekind finite, then \( R \) is Dedekind finite knowing that subrings of Dedekind finite rings are Dedekind finite as well.
\end{proof}

The next necessary and sufficient condition occurs. 

\begin{corollary}
Let \( R \) be a ring. Then, \( R_{h}[x] / \langle x^{n} \rangle \) is Dedekind finite if, and only if, \( R \) is Dedekind finite.
\end{corollary}

Mimicking Ouyang and Chen (see \cite{12}), a ring \( R \) is called \textit{weakly symmetric} if, for all \( a, b, c \in R \), \( abc \in \text{Nil}(R) \) forces \( acb \in \text{Nil}(R) \). Recall also the concept from \cite{3} that a ring \( R \) is \textit{\( J \)-symmetric} if \( abc = 0 \) yields \( bac \in \text{rad}(R) \).

\medskip

We, thereby, come to the following statement.

\begin{lemma}
Let \( R \) be a ring and set \( H := H_{n}(R) \). Then:\\
\text{(1)}\;\; \( H \) is a weakly symmetric ring if, and only if, \( R \) is a weakly symmetric ring.\\
\text{(2)}\;\; \( H \) is a \( J \)-symmetric ring if, and only if, \( R \) is a \( J \)-symmetric ring.
\end{lemma}

\begin{proof}
We only prove the weak symmetric case, as the proof for the \( J \)-symmetric case is quite similar. To that aim, given \( R \) is a weakly symmetric ring, and let \( A = (a_1, a_2, \ldots, a_n) \), \( B = (b_1, b_2, \ldots, b_n) \), \( C = (c_1, c_2, \ldots, c_n) \in H \) such that \( ABC \in \text{Nil}(H) \). Thus, \( a_1b_1c_1 \in \text{Nil}(R) \), and since \( R \) is weakly symmetric, we derive that \( a_1c_1b_1 \in \text{Nil}(R) \). Therefore, \( ACB \in \text{Nil}(H) \), as expected. 

Conversely, if \( H \) is weakly symmetric, then \( R \) is weakly symmetric, as subrings of weakly symmetric rings are weakly symmetric too, as promised.
\end{proof}

Our automatic consequence is this one.

\begin{corollary}
Let \( R \) be a ring. Then:\\
\text{(1)}\;\; \( R_{h}[x] / \langle x^{n} \rangle \) is a weakly symmetric ring if, and only if, \( R \) is a weakly symmetric ring.\\
\text{(2)}\;\; \( R_{h}[x] / \langle x^{n} \rangle \) is a \( J \)-symmetric ring if, and only if, \( R \) is a \( J \)-symmetric ring.
\end{corollary}

We recall that a ring is said to be \textit{abelian} if every its idempotent is central.

\medskip

We are now in a position to show the following.

\begin{lemma}
Let \( R \) be a ring. Then, \( H := H_{n}(R) \) is abelian if, and only if, \( R \) is abelian. In particular, if \( R \) is abelian, then each idempotent in \( H \) is of the form \( (eI_{n}) \), where \( e^{2} = e \in R \).
\end{lemma}

\begin{proof}
Suppose that \( R \) is abelian. We proceed by induction on \( n \). For \( n = 2 \), let \( A = (e_{1}, e_{2}) \) be an idempotent in \( H \). Thus, we receive the next two equalities:
\begin{eqnarray}
e_{1} &=& e_{1}^{2}, \label{4.4} \\
e_{2} &=& e_{1}e_{2} + e_{2}e_{1}. \label{4.5}
\end{eqnarray}
Since \( R \) is abelian, \( e_{1} \) is a central idempotent. Multiplying (\ref{4.5}) on the left side by \( e_{1} \), we can get \( e_{1}e_{2}e_{1} = 0 \), which shows that \( e_{2} = 0 \). Hence, \( A = (e_{1}, 0) \).

Furthermore, for \( n > 2 \), let \( A = ((e_{1}, e_{2}, \ldots, e_{n-1}), e_{n}) \) be an idempotent in \( H \). Then, we arrive at the next two equalities:
\begin{eqnarray}
(e_{1}, e_{2}, \ldots, e_{n-1}) &=& (e_{1}, e_{2}, \ldots, e_{n-1})^{2}, \notag \\
e_{n} &=& e_{1}e_{n} + e_{2}e_{n-1} + \cdots + e_{n-1}e_{2} + e_{n}e_{1}. \label{4.6}
\end{eqnarray}
By the induction hypothesis, \( e_{i} = 0 \) for all \( 2 \leq i \leq n-1 \), and \( e_{1}^{2} = e_{1} \). From (\ref{4.6}), we have \( e_{1}e_{n} + e_{n}e_{1} = e_{n} \), which means that \( e_{n} = 0 \). Thus, \( A = (e_{1}, 0, \ldots, 0) \) is central in \( H \), as required. 

Conversely, if \( H \) is abelian, then \( R \) is abelian, as subrings of abelian rings are always abelian, as requested.
\end{proof}

\begin{corollary}
Let \( R \) be a ring. Then, \( R_{h}[x] / \langle x^{n} \rangle \) is abelian if, and only if, \( R \) is abelian.
\end{corollary}

Next, imitating Nicholson (cf. \cite{11}), a ring \( R \) is called \textit{clean} if any element \( r \in R \) can be written as \( r = u + e \), where \( u \in U(R) \) and \( e^{2} = e \in R \). Recall from \cite{5} that a ring \( R \) is said to be \textit{nil-clean} if, for any \( r \in R \), \( r = q + e \), where \( q \) is nilpotent and \( e^{2} = e \in R \). Also, referring to V{\'a}mos \cite{14}, a ring \( R \) is termed \textit{2-good} if every element in $R$ is a sum of two units of $R$.

\medskip

So, we have the following.

\begin{proposition}
Let \( R \) be a ring and \( H := H_{n}(R) \). Then:\\
\text{(1)}\;\; \( H \) is a clean ring if, and only if, \( R \) is a clean ring.\\
\text{(2)}\;\; \( H \) is a nil-clean ring if, and only if, \( R \) is a nil-clean ring.\\
\text{(3)}\;\; \( H \) is a 2-good ring if, and only if, \( R \) is a 2-good ring.
\end{proposition}

\begin{proof}
We prove the clean case only, because the evidences for the other cases are totally analogous. To that target, suppose \( R \) is clean, and given \( A = (a_{1}, a_{2}, \ldots, a_{n}) \in H \). Since \( R \) is clean, there exist \( u \in U(R) \) and \( e = e^{2} \in R \) such that \( a_{1} = u + e \). Thus, one may decompose \( A = (u, a_{2}, \ldots, a_{n}) + (e, 0, \ldots, 0) \). Apparently, \( (e, 0, \ldots, 0) \) is an idempotent, and \( (u, a_{2}, \ldots, a_{n}) \in U(H) \) in view of Lemma \ref{2.1}. Therefore, \( H \) is clean.

Conversely, assume \( H \) is clean. For each \( r \in R \), we can write \[ (r, 0, \ldots, 0) = (u, \ast, \ldots, \ast) + (e, \ast, \ldots, \ast) ,\] where \( (u, \ast, \ldots, \ast) \in U(H) \) and \( (e, \ast, \ldots, \ast)^{2} = (e, \ast, \ldots, \ast) \). So, \( r = u + e \), where \( u \in U(R) \) and \( e^{2} = e \in R \), completing the arguments.
\end{proof}

We finish our work with the following consequence.

\begin{corollary}
Let \( R \) be a ring. Then:\\
\text{(1)}\;\; \( R_{h}[x] / \langle x^{n} \rangle \) is a clean ring if, and only if, \( R \) is a clean ring.\\
\text{(2)}\;\; \( R_{h}[x] / \langle x^{n} \rangle \) is a nil-clean ring if, and only if, \( R \) is a nil-clean ring.\\
\text{(3)}\;\; \( R_{h}[x] / \langle x^{n} \rangle \) is a 2-good ring if, and only if, \( R \) is a 2-good ring.
\end{corollary}


\medskip
\medskip
\medskip
\medskip
\medskip
\medskip

\noindent{\bf Funding:} The work of the first-named author, P.V. Danchev, is partially supported by the Spanish project Junta de Andaluc\'ia under Grant FQM 264.

\end{document}